\documentclass{amsart}
\usepackage[english]{babel}
\usepackage{esint}
\usepackage{amsfonts,amssymb,amsmath}
\usepackage{setspace}
\usepackage{centernot}
\usepackage{mathrsfs} 
\usepackage{xcolor}
\usepackage{hyperref}

\def \supcom { C_0 ^{\infty} (\erren)}
\newcommand{\erre} {{\mathbb {R}}}
\newcommand{\enne} {{\mathbb {N}}}
\def\erren{{\erre^{ {n} }}}

\def\supp{\mathrm{supp\, }}
\def\inn{\mbox{ in }}
\def\fe {\mbox{ for every  }}
\def\ass{\mbox{ as }}

\def\andd{ \quad\mbox{ and } \quad }

\newcommand{\tende}{\rightarrow}
\newcommand{\ttende}{\longrightarrow}

\newcommand\eps{\varepsilon}
\newcommand\partiali{\partial_{x_i}}

\newcommand\partialj{\partial_{x_j}}
\newcommand {\elle} {{\mathcal {L}}}

\newtheorem{theorem}{Theorem}[section]

\newtheorem{lemma}[theorem]{Lemma}

\theoremstyle{remark}
\newtheorem{remark}[theorem]{Remark}
\theoremstyle{definition}
\newtheorem{definition}[theorem]{Definition}

\theoremstyle{remark}

\numberwithin{equation}{section}

  \title[Asymptotic average solutions for second order hypoeliptic PDEs]{Asymptotic average solutions\\  for second order hypoelliptic PDEs}

\author{Alessia E. Kogoj}
\address{Dipartimento di Scienze Pure e Applicate (DiSPeA)\\ 
				 Universit\`{a} degli Studi di Urbino Carlo Bo\\
				 Piazza della Repubblica 13, 61029 Urbino (PU), Italy.}
\email{alessia.kogoj@uniurb.it}


\begin{document}

\begin{abstract}  
Following an analogous procedure with that used in  \cite{kogoj_lanconelli_pizzetti}, in turn inspired by a 1909 paper by Pizzetti \cite{pizzetti}, we introduce the notion of {\it asymptotic average solutions} for hypoelliptic linear partial differential operators with non-negative characteristic form.
This notion makes every Poisson equation $\elle u(x)=-f(x)$ with continuous data $f$ pointwise solvable.

\end{abstract}
    
\maketitle

\section{Introduction}

 We are dealing with partial differential operators of the type

\begin{equation} \label{operators} 
  \elle = \sum_{i,j=1}^n \partiali (a_{ij}(x) \partialj ) +  \sum_{i=1}^n b_{i}(x) \partiali \inn \erren,
\end{equation}
where the coefficients $a_{ij}$ and  $b_{i}$ are smooth functions in $\erren$. \\ The matrix $A=(a_{ij}), i,j=1,\ldots,n$ is supposed to be symmetric and non-negative definite at any point and not totally degenerate, that is  $\inf_\erren a_{11} >0$.
\\
We also assume that the operator $\elle$ is hypoelliptic and endowed with a non-negative smooth fundamental solution 
$$\Gamma: \{ (x,y)\in\erren \times \erren\  | x \neq y \} \ttende \erre,$$ such that 

\begin{itemize}
\item[{\rm(i)}] for any fixed  $x\in\erren$,
$\Gamma(\cdot,x)$ and
 $\Gamma(x,\cdot)$ belong to  $L^1_{\rm loc}(\erren)$;
 \item[{\rm(ii)}]for every  $\varphi\in\supcom$ and $x\in\erren$, \[\elle
\int_{\erren}
 \Gamma (x,y)\varphi(y)\ d y  \ =\ \int_{\erren}
 \Gamma (x,y)\elle\varphi(y)\ d y \ =\ -\varphi(x);\]

\item[{\rm(iii)}]  $\displaystyle\limsup_{x\tende y}\Gamma(x,y) =\limsup_{y\tende x}\Gamma(x,y) = \infty;$

\item[{\rm(iv)}] $\displaystyle\lim_{|x|\tende \infty} \left(\sup_{|y|\le M} \Gamma(x,y)\right) =\lim_{|y|\tende \infty} \left(\sup_{|x|\le M} \Gamma(x,y)\right) =0$, \\ for every real positive  $M.$
\end{itemize}
Then, given $x\in\erren$ and $r>0$, we define the \textit{$\elle$-ball of center $x$ and radius $r$} as
\begin{equation}\label{lballs}\Omega_r(x) =\left\{y \in\erren\,\,:\,\,\Gamma(x,y)>\frac{1}{r}\right\},\end{equation} and we assume that the following Poisson--Jensen-type  representation formula holds for any function $u\in C^2(\Omega)$, where $\Omega$
is an open subset of $\erren$. For every $\elle$-ball $\Omega_r(x)$ such that $\overline{\Omega_r(x)}\subseteq\Omega$ we have 

 \begin{eqnarray} \label{PJ1} 
 u(x)&=& \frac{1}{r} \int_{\Omega_r(x)} \!\!\! u(y) \frac{ \langle A(y) \nabla_y \Gamma (x,y),  \nabla_y \Gamma (x,y)\rangle}{(\Gamma(x,y))^{2}}\ dy \\ & &\nonumber - \frac{1}{r} \int_0^r \left( \int_{\Omega_\rho(x)} \left(\Gamma(x,y) - \frac{1}{\rho}\right) \elle (u)(y)\  dy \right)d\rho. \end{eqnarray}
We will write  \begin{eqnarray}\label{PJ2}
u(x)=M_r(u)(x)- N_r(\elle u) (x).\end{eqnarray}

All the properties we require are satisfied, in particular, by the classical heat operator, the heat operators on Carnot groups, the Kolmogorov operators studied in \cite{lanconelli_polidoro_1994}, the operators
constructed with the link procedure introduced in \cite{link} and the more general class of hypoelliptic ultraparabolic operators introduced and studied in \cite{kogoj_lanconelli_2004}, which contains all the previous ones.

For this wide class of operators, a fundamental solution with good properties containing the ones required in this work has been constructed in \cite{kogoj_lanconelli_2004} and representation formulae of Poisson--Jensen-type have been proved in  \cite{cintiPJ}. We also have to remark that Kogoj and Tralli have given asymptotic characterizations of the subsolutions in \cite{kt}. Moreover, we mention the works \cite{pascuccipolidoro2006},\cite{boscainpolidoro2007},\cite{kogoj_lanconelli_2007},\cite{cintiuniqueness2009},\cite{cintilanconelliPJ2009}, \cite{CNP2010},\cite{ kogoj_pinchover_polidoro2016},\cite{kogoj_2017}, where Riesz and Poisson-Jensen representation formulae, Gaussian estimates of the fundamental solution, Harnack-type inequalities, Liouville-type theorems, the Cauchy and the Dirichlet problems have been studied for this class of operators.

In this paper, by using the above Poisson--Jensen formula \eqref{PJ1}, we give at first the following definition of  {\it asymptotic average solution} to $\elle u=f$.

 \begin{definition}\label{definizioneuno} Let us define $Q_r(x)=N_r(1)(x)$. Let $u$ and $f$ be continuous functions in an open set $\Omega\subseteq\erren$. We say that $u$ is an {asymptotic average solution} to $\elle u=f$  if 
\begin{equation*} 
 \elle_a u(x):=\lim_{r\tende 0}  \frac{M_r(u)(x)- u(x)}{Q_r(x)} 
\end{equation*}
exists in $\erre$ at every point $x\in\Omega$  and 
\begin{equation*} 
\elle_a u(x)=f(x) \fe  x\in\Omega. \end{equation*}
\end{definition}
This definition completes the one given by Kogoj and Lanconelli in the recent paper \cite{kogoj_lanconelli_pizzetti} - where operators of the second order just in principal part are considered - and the one given by Gutierrez and Lanconelli in \cite{gl2004} for a class of second order operators with non negative characteristic form with drift in the case $f=0$. When $\elle$ is the classical Laplacian, it coincides with the one given by Pizzetti in 1909 in  in his paper on the average of the values that a function takes on the surface of a sphere \cite{pizzetti}.

Furthermore, we will observe that if $u$ is a compactly supported continuous function 
and is an asymptotic average solution to $\elle_a u =f \inn\erren$, then, \begin{center}$\elle_a u =f \inn\erren$ if and only if $ \elle u = f$ in the weak sense of the distributions. \end{center}

In recent years, mean value formulae that characterize classical or viscosity solutions to both linear and nonlinear second order PDEs have been studied by several authors. We refer to the works \cite{gl2004,manfredi2010,manfredi2010tug,manfredi2012,ferrari2015,ferraripinamonti2015,manfredilindqvist2016,magnanini2017,mohammed2021,manfredistroffolini2021,manfredi2022} and the references therein.

Our definition of  asymptotic average solution enables us to solve the Poisson problem $$\elle_a u=-f\inn\erren.$$ 
We state and prove the following Pizzetti-type result. \begin{theorem}\label{mainth} For every $f:\erren\tende\erre$ compactly supported continuous function, the function 
$$u_f(x):=\int_\erren \Gamma(x,y)f(y)\ dy,\qquad x\in\erren.$$
is an asymptotic average solution to 
$$\elle_a u=-f\inn\erren.$$
\end{theorem}
 Pizzetti proved that the Newtonian potentials of continuous bounded functions are solutions of the classical  Poisson-Laplace equations in an asymptotic sense (\cite[Equation (8)]{pizzetti}, see also \cite [Theorem A]{kogoj_lanconelli_pizzetti}). He deduced this result from his celebrated formula (\cite[Equation (A)]{pizzetti}, see also  \cite[Part IV, Chapter 3, pp. 287–288]{CH}) where the spherical integral mean in $\erre^3$ of a smooth function over an Euclidean ball is expressed
 as a power series with respect to the radius of the ball having the iterated Laplacians as coefficients.
While Da Lio and Rodino, \cite{dalio}, generalized Pizzetti’s formula to the heat operator, to the best of our knowledge the result we are considering has not been extended so far in the parabolic setting, even in the case of the heat operator.

To prove our result, we follow an analogous procedure with that used in  \cite{kogoj_lanconelli_pizzetti}.
In Section $2$ we describe some properties of the $\elle$-balls. We will exploit these properties in Section $3$ to prove a Poisson-Jensen type representation formula for the {\it  $\Gamma$-potentials} $u_f$. Thanks to this representation formula, we will be able  to prove in Section $4$ that the Poisson equations relating to the $\elle$ operators are always solvable by the
 {\it  $\Gamma$-potentials}  $u_f$  according to our definition of {asymptotic average solution}. 
 In Section $5$ we compare the notions of classical,  weak and {asymptotic average solutions}; in particular we show that our {asymptotic average solutions} are weak solutions and vice versa.
 
 \section{The $\elle$-balls: some properties}\label{superlevel}

The notion of {asymptotic average solution }
 is based on some representation formulae on the superlevel set of $\Gamma$, the $\elle$-balls defined in \eqref{lballs}.
 From the properties of the fundamental solution, it follows that $\Omega_r(x)$ is a nonempty bounded open set of $\erren$. In addition,
\begin{equation*} 
  \bigcap_{r>0} \Omega_r(x)=\{ x\}
\end{equation*}
and 
\begin{equation*} \frac{|\Omega_r(x)|}{r}\ttende 0 \ass   r\ttende 0,\end{equation*}
where $|\cdot|$ denotes its Lebesgue measure.

Moreover, we prove some lemmata involving $\elle$-ball useful in the subsequent sections.

\begin{lemma}\label{cipollini}  Given a compact subset $G$ of $\erren$ and a positive radius $r$, the set ${G}_r$
\begin{equation*}\label{cipollinichiusi}G_r:=\bigcup_{x\in G} \Omega_r(x),\end{equation*}
 is bounded. 
\end{lemma}

\proof  Let's argue by contradiction and assume $G_r$ unbounded. We can find a sequence $(z_n)$ in $G_r$, 
$$|z_n|\ttende \infty.$$
As $z_n$ belongs to $G_r$,  there exists $x_n\in G$ such that $$z_n\in \Omega_r(x_n)\iff \Gamma(x_n, z_n)>\frac{1}{r}.$$
It follows that for every $n\in\enne$ 
\begin{equation*}
\frac{1}{r}<\Gamma(x_n,z_n)\le \sup_{x\in G} \Gamma(x,z_n).
\end{equation*}
If now we let $n$ go to infinity, by the assumption (iv) related to the fundamental solution $\Gamma$, 
\begin{equation*}
0< \frac{1}{r}\le \lim_{n\tende\infty}\left(  \sup_{x\in G} \Gamma(x,z_n)\right)=0
\end{equation*}
and the proof is complete.
 
 \begin{lemma}\label{palleecipolle}
  \item For every Euclidean ball, centered at any point $x\in \erren$ and with arbitrary radius $R>0$, $B(x,R)$, there exists an $\elle$-ball $\Omega_r(x)$, centered at the same $x$ and with radius $r>0$  such that 
$$\Omega_r(x)\subseteq B(x,R).$$
 \end{lemma}
 \proof  We argue again by contradiction and assume that there exists $R>0$ such that, for every $r>0$, $\Omega_r(x)\nsubseteq B(x,R)$. Choose a sequence $(r_n)$ of real positive numbers decreasing to $0$.  Then, there exists $$y_n\in\Omega_{r_n}(x) \mbox{\ such that\ } y_n\notin B(x,R).$$ 
This means
$$\Gamma(x,y_n)>\frac{1}{r_n}\andd y_n\notin B(x,R).$$
 By the assumption (iii) that we require on $\Gamma$,
$\Gamma(x,y)\ttende 0$ as $y\ttende  \infty$. Consequently, the sequence $(y_n)$ has to be bounded. We assume 
$$\lim_{n\to  \infty} y_n=y^\ast$$ for a suitable $y^\ast\in\erren.$  Then  $y^\ast\notin B(x,R)$ so that, being $y^*\neq x$,  $\Gamma(x,y^\star)< \infty.$ On the other hand, $$\Gamma(x,y^\ast)=\lim_{n\tende \infty} \Gamma(x,y_n)\geq \lim_{n\tende \infty}  
\frac{1}{r_n}= \infty.$$ This contradiction concludes the proof.

\section{A representation formula}
Let $K$ be a compact subset of $\erren$ and let $f:\erren\ttende\erre$ be a continuous function with $\supp f \subseteq K$.  Define 
$$u_f(x):=\int_\erren \Gamma(x,y)f(y)\ dy,\qquad x\in\erren.$$

We want to prove that $u_f$ satisfies the following Poisson--Jensen type representation formula for every $x\in\erren$:
\begin{equation} \label{imp!} 
u(x)=M_r(u)(x)+N_r(f)(x),
\end{equation}
where $M_r$ and $N_r$ are defined as in \eqref{PJ2}.

We start the proof  by choosing  a sequence  of smooth functions $(f_p)$ in $\erren$  such that
$\supp f_p\subseteq K$ for every $p\in\enne$ and 
$$\sup_K |f_p-f|\ttende 0\ass p\ttende\infty.$$
Then, we define 
\begin{eqnarray*}
u_{f_p}(x)=\int_\erren \Gamma(x,y)f_p(y)\ dy = \int_K\Gamma(x,y) f_p(y)\ dy.
\end{eqnarray*} 
We observe at first that
$$u_{f_p}\in C^\infty (\erren,\erre) \mbox{  for every  } p\in\enne.$$
Moreover, 
$$\elle u_{f_p}=-f_p.$$
Then, 
 the Poisson--Jensen-type representation formula \eqref{PJ1} holds for $u_{f_p}$ and one has 
\begin{eqnarray*}  
u_{f_p}(x)&=&M_r(u_{f_p})(x)-N_r(\elle u_{f_p})(x)\\&=&M_r(u_{f_p})(x)+N_r(f_p)(x),
\end{eqnarray*}
for every $p\in\enne.$

By the Lebesgue dominated convergence Theorem, for every $x\in\erren,$
\begin{eqnarray*}
\lim_{p\tende  \infty} u_{f_p}(x) =\int_K \Gamma(x,y)\lim_{p\tende\infty} f_p(y)\ dy = u_f(x).
\end{eqnarray*} 
So, in order to have the  \eqref{imp!}, it remains to show that,  for every $x\in\erren$,
\begin{eqnarray}\label{2.5} 
\lim_{p\tende \infty} M_r(u_{f_p})(x)=M_r(u_f)(x)
\end{eqnarray} 
and
\begin{eqnarray}\label{2.6} 
\lim_{p\tende  \infty} N_r(f_p)(x)=N_r(f)(x).
\end{eqnarray}

\proof[Proof of \eqref{2.5}]

We observe that, for every compact set $G\subseteq\erren$,
\begin{equation*}
\sup_G |u_{f_p}-u_f|\ttende 0\ass p\ttende \infty. 
\end{equation*} Indeed,
\begin{eqnarray*} \sup_G|u_{f_p}-u_f| &\le& \sup_{x\in G} \left| \int_K \Gamma(x,y) (f_p(y)- f(y))\ dy \right|  \\
&\le& \sup_{K} |f_p-f| \sup_{x\in G}\int_K \Gamma(x,y)\ dy \\
&=& C(G,K) \sup_K |f_p-f|.
\end{eqnarray*} 
Furthermore, for every $x\in\erren$ we have
\begin{eqnarray*}
 | M_r(u_{f_p})(x) -M_r(u_f)(x)|&= &| M_r(u_{f_p}-u_f)(x)|\\&\le& \sup_{\Omega_r(x)} |u_{f_p}-u_f| M_r(1)(x)\\&=&  \sup_{\Omega_r(x)} |u_{f_p}-u_f|\\&\le&  \sup_{\overline {\Omega_r(x)}} |u_{f_p}-u_f|.
\end{eqnarray*} 
So, as ${\overline {\Omega_r(x)}}$ is compact (see Lemma \ref{cipollini}), when $p$ goes to  $ \infty$ the right hand side goes to zero  and 
\begin{eqnarray*} | M_r(u_{f_p})(x) -M_r(u_f)(x)|\ttende 0.\end{eqnarray*}

\proof[Proof of \eqref{2.6}] We observe at first that if $G$ is a compact subset of $\erren$ and  $r>0$, there exists a positive constant $C(r,G)$ such that
\begin{equation}\label{2.2} 
\sup_{x\in G} Q_r(x)\le C(r,G).
\end{equation}
Indeed for every $x\in G$ we have

\begin{eqnarray*} \nonumber
 Q_r(x) &=& \frac{1}{r} \int_0^r \left( \int_{\Omega_\rho(x)} \Gamma(x,y) - \frac{1}{\rho} \  dy \right)\ d\rho \\  \nonumber
 &\le& \frac{1}{r} \int_0^r \left( \int_{\Omega_\rho(x)} \Gamma(x,y)  \  dy \right) \ d\rho \\
 &\le& \frac{1}{r} \int_0^r \left( \int_{G_\rho} \Gamma(x,y)  \  dy \right) \ d\rho,\end{eqnarray*}
where $G_\rho$ is the bounded set defined in Lemma \ref{cipollini}.
Take now a non-negative $\varphi \in C_0^\infty(\erren,\erre)$ such that $\varphi=1$ on $G_\rho$. For every $x\in G$ we have 
\begin{eqnarray*} \int_{G_\rho} \Gamma(x,y)\ dy &\le& \int_{\erren} \varphi(y)\Gamma(x,y)\ dy 
\\ &\le&\sup_{x\in G}  \int_{\erren} \varphi(y)\Gamma(x,y)\ dy 
\\ &=& C(\varphi,G).
\end{eqnarray*}
Thus,
\begin{eqnarray*}  \sup_{x\in G} Q_r(x) \le C(\varphi,G).\end{eqnarray*}
Now, for every $x\in\erren$, we have 
\begin{eqnarray*} | N_r(f_p)(x) -N_r(f)(x)|&\le&| N_r(f_p-f)(x)| \\&\le& \sup_{K} |f_p-f| N_r(1)(x).
\end{eqnarray*} 
Then, for every compact set $G\subseteq\erren$, recalling that we defined $Q_r(x)=N_r(1)(x)$, 
\begin{eqnarray*} \sup_G | N_r(f_p) -N_r(f)|&\le& \sup_{K} |f_p-f| \sup_{x\in G} |Q_r(x)|\\ &\le& C(\varphi, G) \sup_K|f_p-f|.
\end{eqnarray*} 
So we have proved that $(N_r(f_p))$ is uniformly convergent to $N_r(f)$ on every compact subset of $\erren$ and, in particular,   that \eqref{2.6} holds.

\section{A Pizzetti-type theorem} In this section we prove our main result (Theorem \ref{mainth}).
It will follow easily by the representation formula of the previous section and the following result.

\begin{lemma}\label{corollary24?} Let $f:\erren\ttende\erre$ be a continuous function. Then, for every $x\in\erren,$ 
$$\sup_{y\in\Omega_r(x)} |f(y)-f(x)| \ttende 0 \ass r\ttende 0.$$
\end{lemma}
\begin{proof} Since $f$ is continuous at $x$, for every $\eps >0$ there exists $R>0$ such that 
$$\sup_{y\in B(x,R)} |f(y)-f(x)| <\eps.$$
By Lemma \ref{palleecipolle}, there exists $r_0>0$ such that $\Omega_{r_0}(x)\subseteq B(x,R).$ Then, for every $r<r_0$, 
$$\sup_{y\in\Omega_r(x)} |f(y)-f(x)| \le \sup_{y\in\Omega_{r_0}(x)} |f(y)-f(x)| \le \sup_{y\in B(x,R)} |f(y)-f(x)| <\eps.$$
We have so proved that for every $\eps>0$ there exists $r_0>0$ such that 
$$\sup_{y\in \Omega_r(x)} |f(y)-f(x)| <\eps$$ for every $r<r_0.$  Hence,
$$\lim_{r\tende0 } \left( \sup_{y\in\Omega_r(x)} |f(y)-f(x)|\right) = 0.$$
\end{proof}

Finally, we can give the proof of Theorem \ref{mainth}. Let $f:\erren\ttende\erre$ be a continuous function with compact support and 
$$u(x)=\int_\erren \Gamma(x,y)f(y)\ dy,\qquad x\in\erren.$$ We use the representation formula 
\eqref{imp!} and for every $x\in\erren$ we have 

\begin{eqnarray*} 
\frac{M_r(u)(x)-u(x)}{Q_r(x)} =  - \frac{N_r(f)(x)}{Q_r(x)}  
\end{eqnarray*} so that, as $f(x)$ is constant with respect to  $y\in \Omega_r(x)$,
\begin{eqnarray*} 
\left|\frac{M_r(u)(x)-u(x)}{Q_r(x)} +  f(x) \right| &= & 
\frac{1}{Q_r(x)} \left| {N_r(f(x)- f})(x)\right| \\ &\le &  \sup_{y\in\Omega_r(x)} |f(x)-f(y)|.
 \end{eqnarray*} 
 
 Letting $r$  go to $0$, by Lemma \ref{corollary24?} the left hand side goes to zero 
 and we obtain
 \begin{eqnarray*} 
 \lim_{r\tende 0} \frac{M_r(u)(x)-u(x)}{Q_r(x)} = -  f(x) 
 \end{eqnarray*} 
 for every $x\in\erren$. 
 
\section{Classical, weak and asymptotic average solutions}

It is well known that the Poisson equation $\elle u(x)=-f(x)$ with continuous data $f$ does not admit, in general, a solution in the classical sense. 
Clearly, if the Poisson problem admitted a solution $u$ in the classical sense, by the Poisson--Jensen-type formula \eqref{PJ1} and the same arguments of the previous section, one can see that  $u$ verifies the following 
\begin{eqnarray*} 
 \lim_{r\tende 0} \frac{M_r(u)(x)-u(x)}{Q_r(x)} = -  f(x) 
 \end{eqnarray*} 
and so it is also a solution in the asymptotic sense.

It is actually this formula that suggests the definition of asymptotic average solution for less regular functions whenever the above limit exists.
 
When the solution $u$ is a continuous function with compact support, the notion of weak solution and  {asymptotic average solution} are equivalent. In particular, we have the following theorem.
\begin{theorem}\label{consequence} Let $f,u:\erren\ttende\erre$ be  compactly supported continuous functions. Then, 
$$\elle_a u =- f  \inn\erren$$ if and only if 
$$\elle u = - f\inn \mathcal{D}'(\erren).$$
\end{theorem}

\proof

By the definition of  {asymptotic average solution} in Theorem \ref{mainth},
$$\elle_a u=-f \inn\erren$$ if and only if 
$$\elle_a(u-u_f)=0\inn\erren.$$
Then, thanks to Corollary 3.4 in  \cite{gl2004}, we know that $u-u_f\in C^\infty(\erren,\erre)$ and 
$$\elle(u-u_f)=0$$
in the classical sense (and vice versa). Since $\elle$ is hypoelliptic, this is equivalent to
$$\elle(u-u_f)=0 \inn  \mathcal{D}'(\erren),$$that, in turn, is equivalent to 
\begin{equation*}\label{1.3}
\elle(u)=\elle(u_f) \inn  \mathcal{D}'(\erren).
\end{equation*}
On the other hand, as $\Gamma$ is a fundamental solution of $\elle$, $\elle(u_f)=-f\inn \mathcal{D}'(\erren)$.  Thus, 
$$\elle u=-f  \inn  \mathcal{D}'(\erren).$$ 
The proof is completed in view of the following remark.
\begin{remark}
Corollary 3.4 in  \cite{gl2004} requires that two conditions (2.1 and  2.2 in that paper), that are related to our operators $\elle$, be satisfied.  
The first one, Condition 2.1,  requires that
{\it for every bounded open set $V$ there exists a function  $h\in C^2(V)$ such that  $\elle h <0$ in  $V$ and $\inf_{V}  h >0.$  }
In our case, it is enough to choose two  suitable positive real constants, $K$ and
$\lambda$, and a function  $h(x):= K - e^{\lambda x_1}$. It is well known that this condition implies the following maximum principle due to Picone.\\
{\it For every  $u\in C^2 (V)$ such that \[\elle u\geq 0\quad\mbox { in } 
V,\quad\limsup_{x\tende y }u(x) \le 0\quad \forall y \in\partial V, \]  we have $u\le 0$ in  $V$.}

Condition 2.2, i.e. {\it the existence of a basis of the Euclidean topology formed by regular sets},
can be proved as in \cite[Corollarie 5.2]{Bony}, see also \cite[Proposition 7.1.5] {BLU}. The tools used in its proof are the hypoellipticity, the non totally degeneracy of the operator $\elle$  and the aforementioned Picone maximum principle. 
\end{remark}

\section*{Acknowledgement} The author is a member of the Gruppo Nazionale per l'Analisi Matematica, la Probabilit\`a e le loro Applicazioni (GNAMPA) of the Istituto Nazionale di Alta Matematica (INdAM). 

\section*{Declarations}  

\begin{itemize}

\item[-] Conflict of interest: The authors declare no conflict of interest.
\item[-]  Data availability:  Data sharing not applicable to this article as no datasets were generated or analysed during
the current study.

\item[-]   Funding: No funding was received.

\end{itemize}

\end{document}